\documentclass{amsart}
\usepackage{helvet}
\usepackage{setspace}
\usepackage{enumitem}
\usepackage[all, cmtip]{xy}
\usepackage{amssymb}
\usepackage{graphicx}
\usepackage{ragged2e}
\usepackage[
    style=numeric,
    sorting=nyt,  
    maxnames=5,
    minnames=1,
    giveninits=true,
    doi=true,
    url=true,
    backend=biber
]{biblatex}
\DefineBibliographyExtras{english}{

}

\DeclareNameAlias{sortname}{family-given}
\DeclareNameAlias{default}{family-given}

\AtBeginBibliography{
  
}
\DeclareFieldFormat[article]{title}{#1}
\DeclareFieldFormat[book]{title}{\mkbibemph{#1}}
\DeclareFieldFormat{journaltitle}{#1}

\usepackage{xcolor} 
\definecolor{darkblue}{RGB}{0,0,139} 
\usepackage{etoolbox} 
\usepackage[margin=1in]{geometry} 
\usepackage{tikz}
\usepackage{amsmath,amsthm,amssymb}
\usepackage[colorlinks=true,  
            linkcolor=darkblue,  
            citecolor=darkblue, 
            urlcolor=darkblue 
           ]{hyperref}
\usepackage{abstract}

\usepackage{titlesec}
\titleformat{\section}
  {\normalfont\Large\bfseries\color{darkblue}}
  {\thesection}{1em}{}
\titleformat{\subsection}
  {\normalfont\large\bfseries\color{darkblue}}
  {\thesubsection}{1em}{}
\titleformat{\subsubsection}
  {\normalfont\normalsize\bfseries\color{darkblue}}
  {\thesubsubsection}{1em}{}

\newcommand{\N}{\mathbb{N}}

\newtheoremstyle{plainnormal}  
  {3pt}      
  {3pt}      
  {\normalfont} 
  {}         
  {\color{darkblue}\bfseries} 
  {.}        
  {0.5em}    
  {}         

\theoremstyle{plainnormal}
\newtheorem{theorem}{Theorem}[section]
\newtheorem{lemma}[theorem]{Lemma}
\newtheorem{corollary}[theorem]{Corollary}
\newtheorem{proposition}[theorem]{Proposition}
\newtheorem{definition}[theorem]{Definition}
\newtheorem{example}[theorem]{Example}

\newtheorem{remark}[theorem]{Remark}

\makeatletter
\renewenvironment{proof}[1][\proofname]{
  \par\pushQED{\qed}\normalfont
  \topsep6pt \partopsep0pt
  \trivlist
  \item[\hskip\labelsep\color{darkblue}\itshape #1.]
}{
  \popQED\endtrivlist\@endpefalse
}
\makeatother
\usepackage{cleveref}

\begin{document}
\onehalfspacing
\title{Weakly Sigma-cotorsion rings}
\author{Manuel Cortés-Izurdiaga \and Sergio Estrada \and José Manuel Fresneda}
\date{}

\thanks{M.C.I. and J.M.F. were partly supported by grant PID2024-158993NB-100. S.E. and J.M.F. were partly supported by grant PID2024-155576NB-I00. Both grants are funded by MICIU/AEI/10.13039/501100011033 /FEDER, UE.}

\subjclass[2020]{16E10, 16D90, 16L30}
\keywords{left $n$-perfect ring, left weakly $\Sigma$-cotorsion ring, $n$-$\Sigma$-cotorsion module, direct sums of injective modules}

\maketitle
\begin{center}
\begin{minipage}{0.8\textwidth}
\textbf{\textcolor{darkblue}{Abstract.}} 
\justifying
We study the class of rings $R$ for which every direct sum of injective $R$-modules is cotorsion. We call them weakly $\Sigma$-cotorsion rings. The defining property might be seen as the dual of Chase's characterization of coherence in terms of the flatness of every direct product of projective $R$-modules. More generally, we study rings over which direct sums of injective modules have finite cotorsion dimension and call them weakly $n$-$\Sigma$-cotorsion rings, as well as rings over which direct sums of cotorsion modules have finite cotorsion dimension (called $n$-$\Sigma$-cotorsion rings). In the process, we obtain new characterizations of $n$-perfect rings and extend previous results by Guil Asensio and Herzog, and by Šaroch and Šťovíček.

\end{minipage}
\end{center}

\section{Introduction}

Coherent and perfect rings were introduced in 1960 in the seminal works of Chase \cite{Chase1960} and Bass \cite{Bass1960finitistic}, respectively. 
With respect to coherence, let $R$ be any ring. As a consequence of Chase's Theorem \cite[Theorem 2.1]{Chase1960}, right coherent rings can be defined by the property that the direct product of any family of projective left $R$-modules is flat. In fact, it is enough to test this for direct product of any family of copies of $R$. In the language of model theory, right coherent rings can be characterized with the property that the definable closure of the class $\mathrm{Proj}$ of projective left $R$-modules is contained in (and in fact coincides with) the class $\mathrm{Flat}$ of all flat left $R$-modules.

In turn, left perfect rings are characterized with the property that the classes of $\mathrm{Flat}$ and $\mathrm{Proj}$ agree. Now, flat left $R$-modules are the left part of the so-called flat-cotorsion pair $(\mathrm{Flat}, \mathrm{Cot})$. Here, $\mathrm{Cot}$ denotes the class of cotorsion left $R$-modules, that is, left $R$-modules $C$, such that $\mathrm{Ext}^{1}_R(F,C) = 0$, for any flat left $R$-module $F$. The flat-cotorsion pair was a crucial tool in the proof of the Flat Cover Conjecture (see Bican, El Bashir and Enochs \cite{BicanElBashirEnochs2001}). With this terminology, it is well known that a ring $R$ is left perfect if and only if $\mathrm{Cot}=R\mathrm{\text{-}Mod}$ (the class of all left $R$-modules). In  \cite[Theorem~19]{GuilAsensioHerzog2005}, Guil Asensio and Herzog proved that if a ring $R$ is such that the class $\mathrm{Cot}$ is closed under arbitrary direct sums, then $R$ must be left perfect, thus providing a new characterization of left perfect rings. 

The aim of this paper is to study closure properties under direct sums of the classes $\mathrm{Inj}$ of injective left $R$-modules and $\mathrm{Cot}$, with a twofold motivation. On the one hand, our results specialize to yield a statement that is dual in nature to Chase's original equivalent conditions for coherence mentioned above. On the other hand, we extend the characterization of left perfect rings due to Guil Asensio and Herzog to the broader class of so-called \emph{left $n$-perfect rings}, that is, rings for which every flat left $R$-module has projective dimension at most $n$. We note that this property is also referred in the literature as $\mathrm{splf}(R)\leq n$, where $\mathrm{splf}(R)$ denotes the supremum of the projective lengths of flat left $R$-modules. 

The second of these questions is addressed in Section \ref{section:When direct sums of cotorsion have finite cotorsion dimension}, in  Theorem \ref{theorem:Generalization_Pedro_Ivo}. We say that a left $R$-module $M$ is \emph{$n$-$\Sigma$-cotorsion} if, for every index set $I$, the direct sum $M^{(I)}$ has cotorsion dimension at most $n$. Furthermore, a ring $R$ is called a \emph{left $n$-$\Sigma$-cotorsion ring} if $R$ is $n$-$\Sigma$-cotorsion as a left $R$-module. In the particular case $n = 0$, we recover the classical notions of $\Sigma$-cotorsion modules and left $\Sigma$-cotorsion rings (\cite{GuilAsensioHerzog2005}).

\medskip\par\noindent
\textbf{Characterizing left $n$-perfect rings.}
Let $n \geq 0$ be an integer. For any ring $R$, the following statements are equivalent:
\begin{enumerate}[label=(\arabic*)]
\item $R$ is left $n$-perfect.
\item Every cotorsion left $R$-module is $n$-$\Sigma$-cotorsion.
\item The cotorsion envelope $C(R)$ of $R$ as a left $R$-module is $n$-$\Sigma$-cotorsion.
\item $R$ is left $n$-$\Sigma$-cotorsion.
\item The pure-injective envelope $PE(R)$ of $R$ as a left $R$-module is $n$-$\Sigma$-cotorsion.
\end{enumerate}
A key ingredient in the proof of this theorem is a generalization of a result due to Šaroch and Šťovíček \cite[Theorem~3.3]{sarich2020singular}; see Theorem~\ref{theorem:Generalization_S-S}.

Now, if we relax condition (2) above by requiring only that every \emph{injective} left $R$-module is $n$-$\Sigma$-cotorsion (equivalently, that every direct sum of left injective $R$-modules has finite cotorsion dimension at most $n$), we obtain the following result (see Theorem \ref{theorem:Caracterization_sums_of_injectives_are_in_Cot_n}):

\medskip\par\noindent
\textbf{Characterizing rings for which direct sums of injective modules have finite cotorsion dimension.}
    Let $n\geq0$ be an integer. Then, the following assertions are equivalent over any ring $R$:
    \begin{enumerate}[label=(\arabic*)]
        \item Every injective left $R$-module is $n$-$\Sigma$-cotorsion.
        \item The left $R$-module $R^{+}$ is $n$-$\Sigma$-cotorsion, where $R^{+}$ is the character module of $R_R$.
        \item Any left $R$-module in the definable closure of the injective left $R$-modules has finite cotorsion dimension at most $n$.
        \item Every $\mathrm{FP}$-injective $R$-module has finite cotorsion dimension at most $n$.
    \end{enumerate}

A ring satisfying any of the equivalent conditions above will be called \emph{left weakly $n$-$\Sigma$-cotorsion}. Such rings can be viewed as the injective-dual counterparts of the class of rings $R$ for which every direct product of projective left $R$-modules has finite flat dimension. The latter were studied by Cortés-Izurdiaga in \cite{cortes2016products} under the name of \emph{right weak $n$-coherent rings}.

The case $n=0$ (Corollary \ref{corollary:Caracterization_sums_of_injectives_are_cotorsion}) is of special interest, as it provides a natural analogue to the usual characterizations of right coherent rings described above. In this case, we will simply say that $R$ is \emph{left weakly $\Sigma$-cotorsion}.

\medskip\par\noindent
\textbf{Chase's dual characterization for direct sums of injectives.}
The following properties of a ring $R$ are equivalent:
\begin{enumerate}
\item $R$ is left weakly $\Sigma$-cotorsion, that is, any direct sum of injective left $R$-modules is cotorsion.
\item The left $R$-module $R^{+}$ is $\Sigma$-cotorsion, where $R^{+}$ is the character module of $R_R$..
\item The definable closure of the injective left $R$-modules is contained in $\mathrm{Cot}$.
\item Every $\mathrm{FP}$-injective $R$-module is cotorsion.
\end{enumerate}

In addition, in Proposition \ref{proposition:Consequences_of_sum_of_injectives_is_cotorsion}, we collect several module-theoretic properties of left weakly $\Sigma$-cotorsion rings.

We would like to stress that left weakly $\Sigma$-cotorsion rings have already appeared implicitly in the literature. In particular, Estrada, Pérez and Zhu \cite[Lemma~5.1]{EstradaPerezZhu2020}, working in the context of relative homological algebra, showed that, using the present terminology, if $\mathrm{Flat}$ forms the left-hand side of a balanced pair (see \cite{EstradaPerezZhu2020} for the terminology), then a ring $R$ is left weakly $\Sigma$-cotorsion if and only if it is left perfect (equivalently, $R$ is left $\Sigma$-cotorsion). We extend this result in Corollary \ref{corollary:Generalization_lemma_5.1(1)} for left weakly $n$-$\Sigma$-cotorsion rings.

The paper is completed by providing an abundant collection of families of examples and non-examples of the classes of rings introduced. Of course, left Noetherian rings and left perfect rings are examples of left weakly $\Sigma$-cotorsion rings. However, there are many examples of left weakly $\Sigma$-cotorsion rings that are neither left Noetherian nor left perfect; we give specific examples in the setting of formal triangular matrix rings and path rings (see Example~\ref{examples:weakly_Sigma-cotorsion}(2) and Example~\ref{examples_with_matrix_rings}(2)). Note that the existence of such non-noetherian non-perfect left weakly $\Sigma$-cotorsion ring allows one to extend the lack of flat balance in \cite[Theorem 5.2]{EstradaPerezZhu2020} and \cite[Theorem 4.1]{Enochs2015Balance} beyond the realm of noetherian rings.

Moreover, every left $n$-perfect ring is left weakly $n$-$\Sigma$-cotorsion, and every left weakly $r$-$\Sigma$-cotorsion ring is left weakly $n$-$\Sigma$-cotorsion, for all $0\leq r\leq n$. Nevertheless, among others, a classical example of Kaplansky \cite{Kaplansky1958_Hereditary} (see Example \ref{examples:weakly_Sigma-cotorsion}(3)) shows that there exist left weakly $n$-$\Sigma$-cotorsion rings with $n>0$ that are not left weakly $\Sigma$-cotorsion.

Finally, using Nagata's construction (\cite[Appendix A1, Example 1]{Nagata1962_LocalRings}), we exhibit in Example \ref{example:Nagata}(1) a left weakly $\Sigma$-cotorsion ring that is not left $n$-perfect for any $n\geq 0$. This shows that the class of left $n$-perfect rings is strictly contained in the class of left weakly $n$-$\Sigma$-cotorsion rings.


\section{Preliminaries}
Throughout this paper, $R$ denotes a (not necessarily commutative) ring with identity. Unless stated otherwise, all modules are assumed to be left $R$-modules. Whenever right $R$-modules are considered, this will be indicated explicitly. We denote by $\mathrm{Flat}$, $\mathrm{Inj}$, and $\mathrm{Cot}$ the classes of flat,  injective, and cotorsion $R$-modules, respectively. For any integer $n \geq 0$, we write $\mathrm{Cot}_n$ for the class of $R$-modules of cotorsion dimension at most $n$, noting that $\mathrm{Cot}_0 =\mathrm{Cot}$. Recall that the \emph{character module} of a $R$-module $M$ is the right $R$-module $M^+ := \mathrm{Hom}_{\mathbb{Z}}(M, \mathbb{Q}/\mathbb{Z})$, and that the definition for right $R$-modules is analogous. 

We now recall some standard notions from homological algebra. 
Given a class of $R$-modules $\mathcal{C}$, we denote by $\mathcal{C}^{\perp}$ its right orthogonal class, that is, the class of $R$-modules $X$ such that $\mathrm{Ext}^1_R(C,X) = 0$ for all $C \in \mathcal{C}$. 
The left orthogonal class is defined dually. 
A pair of classes $(\mathcal{C},\mathcal{L})$ is called a \emph{cotorsion pair} if $\mathcal{C}^{\perp} = \mathcal{L}$ and $^{\perp}\mathcal{L} = \mathcal{C}$. 
Such a cotorsion pair is \emph{complete} if, for every $R$-module $M$, there exist exact sequences
\[
0 \longrightarrow L \longrightarrow C \longrightarrow M \longrightarrow 0 \quad \text{and} \quad
0 \longrightarrow M \longrightarrow L' \longrightarrow C' \longrightarrow 0,
\]
with $C,C' \in \mathcal{C}$ and $L,L' \in \mathcal{L}$, and it is \emph{hereditary} if $\mathrm{Ext}^i_R(C,L) = 0$ for all $i \geq 1$, $C \in \mathcal{C}$, and $L \in \mathcal{L}$.

We now turn to notions related to definable classes of modules. 
A class of $R$-modules $\mathcal{D}$ is called \emph{definable} if it is closed under direct products, direct limits, and pure submodules. 
For a class $\mathcal{C}$ of $R$-modules, we denote by $\mathrm{Cogen}(\mathcal{C})$ the class of modules cogenerated by $\mathcal{C}$, that is, the class of all submodules of products of modules from $\mathcal{C}$. Analogously, $\mathrm{Cogen}_*(\mathcal{C})$ denotes the class of all pure submodules of products of modules from $\mathcal{C}$, and $\langle \mathcal{C} \rangle$ denotes the \emph{definable closure} of $\mathcal{C}$, i.e., the smallest definable class containing $\mathcal{C}$. If $\mathcal{C} = \{C\}$, we write simply $\mathrm{Cogen}(C)$, $\mathrm{Cogen}_*(C)$, and $\langle C \rangle$, respectively.

It is well known that every definable class is closed under pure-epimorphic images \cite[Theorem~3.4.8]{Prest2009Purity}. 
Moreover, the definable closure $\langle \mathcal{C} \rangle$ can be constructed by first closing $\mathcal{C}$ under products, then under pure submodules, and finally under pure-epimorphic images \cite[Lemma 2.9]{herbera2014definable}. Furthermore, for any definable class $\mathcal{D}$, there exists a pure-injective module $C \in \mathcal{D}$ such that $\mathrm{Cogen}_*(C) = \mathcal{D}$ \cite[Corollary~5.3.52]{Prest2009Purity}; such a $R$-module is referred to as an \emph{elementary cogenerator} of $\mathcal{D}$. 

We finish this section with some facts related to (infinitely) presented $R$-modules. Let $\lambda$ be a cardinal; we denote by $\lambda^+$ the next cardinal of $\lambda$. We say that $\lambda$ is \textit{singular} if it is equal to the sum of a family consisting of less than $\lambda$ cardinals which are strictly smaller than $\lambda$. Otherwise, the cardinal is \textit{regular}.

Given an infinite regular cardinal $\lambda$ and a $R$-module $M$, we say that $M$ is $<\!\lambda$-\emph{generated} if it admits a generating set of cardinality strictly smaller than $\lambda$. 
Similarly, $M$ is $<\!\lambda$-\emph{presented} if there exists an exact sequence $R^{(\gamma)} \longrightarrow R^{(\mu)} \longrightarrow M \longrightarrow 0$, with cardinals $\mu, \gamma < \lambda$. A short exact sequence
\begin{displaymath}
    0 \rightarrow M \xrightarrow{f} L \xrightarrow{g} N \rightarrow 0,
\end{displaymath}
is \textit{$\lambda$-pure} if every $<\lambda$-presented module is projective with respect to it. In this case, $f$ is called a \textit{$\lambda$-pure monomorphism} and $g$, a \textit{$\lambda$-pure epimorphism}. If $\lambda=\aleph_0$, we obtain the usual pure-exact sequences.

Since some of our results involve $\lambda$-continuous directed systems, we next recall a few definitions from \cite{sarich2020singular} that will be used throughout the paper: 
\begin{itemize}
    \item For a regular uncountable cardinal $\lambda$, a directed system 
    \[
    \mathcal{M} = \left( M_i, f_{ji}: M_i \to M_j \mid i<j \in I \right)
    \] 
    of modules is called \emph{$\lambda$-continuous} if the poset $(I,\leq)$ has suprema of all chains of length $<\lambda$, 
    and for any such chain $J \subseteq I$, one has $M_{\sup J} = \varinjlim_{j \in J} M_j$.
    \item Let $M$ be a $R$-module and $\mathcal{C}$ a class of $R$-modules. 
    We say that $M$ is \emph{almost $(\mathcal{C},\lambda)$-projective} if $M$ is the direct limit of a $\lambda$-continuous directed system $\mathcal{S}$ consisting of $<\!\lambda$-presented modules from $^{\perp}\mathcal{C}$.
\end{itemize}

                    
\section{\texorpdfstring{$n$-$\Sigma$}{n-Sigma}-cotorsion modules}

Let $M$ be an $R$-module. We provide a characterization of when every direct sum of copies of $M$ has finite cotorsion dimension. To achieve this, we extend \cite[Theorem~3.3]{sarich2020singular} to the broader class $\mathrm{Cot}_n$ for every integer $n \geq 0$. This generalization is motivated by the goal of understanding the structure of $R$-modules with bounded cotorsion dimension and of providing a framework for subsequent characterizations of rings in terms of such $R$-modules. 

\begin{definition} \label{def:n-Sigma-cotorsion-module}
Let $M$ be an $R$-module, and let $n \geq 0$ be an integer. We say that $M$ is \emph{$n$-$\Sigma$-cotorsion} if every direct sum of copies of $M$ belongs to $\mathrm{Cot}_n$. In the case $n=0$, $M$ is simply called \emph{$\Sigma$-cotorsion}.
\end{definition}

\begin{lemma}\label{lemma:nSigma_cotorsion_equiv}
Let $n \geq 0$ be an integer, and let $\mathcal{C}$ be a class of $R$-modules that is closed under direct products. Then the following statements are equivalent:
\begin{enumerate}[label=(\arabic*)]
    \item Every $R$-module in the class $\mathcal{C}$ is $n$-$\Sigma$-cotorsion.
    \item For every family of $R$-modules $\{C_i \mid i \in I\} \subseteq \mathcal{C}$, the direct sum $\bigoplus_{i \in I} C_i$ belongs to $\mathrm{Cot}_n$.
\end{enumerate}
\begin{proof}
Clearly, (2) $\Rightarrow$ (1). Thus, it suffices to prove that (1) $\Rightarrow$ (2). Let $\{C_i \mid i \in I\} \subseteq \mathcal{C}$ be a family of $R$-modules, and set $T := \prod_{i \in I} C_i$. By hypothesis, $T \in \mathcal{C}$. Observe that each $C_i$ is a direct summand of $T$. Consequently, the $R$-module $\bigoplus_{i \in I} C_i$
is a direct summand of $T^{(I)}$. By (1), the module $T$ is $n$-$\Sigma$-cotorsion. Since the class $\mathrm{Cot}_n$ is closed under direct summands (\cite[Proposition 19.2.1]{MaoDing2006CotorsionDimension}), it follows that $\bigoplus_{i \in I} C_i$ belongs to $\mathrm{Cot}_n$.
\end{proof}
\end{lemma}

The following class of modules will play a central role in the study of $\mathrm{Cot}_n$:

\begin{definition} \label{definition_of_C_n}
Let $n \geq 1$ be an integer. 
We denote by $\mathcal{C}_n$ the class of all $R$-modules $K$ for which there exists an exact sequence
\[
0 \longrightarrow K \longrightarrow L_{n-1} \longrightarrow L_{n-2} \longrightarrow \cdots \longrightarrow L_0 \longrightarrow M \longrightarrow 0,
\]
where each $L_i$ is a free $R$-module and $M$ is a flat $R$-module. 
Note that the class $\mathcal{C}_n$ coincides with the class of all $R$-modules $K$ for which there exists an exact sequence as above in which the modules $L_i$ are merely projective (rather than free).
\end{definition}
To establish our main result, we first develop several preliminary tools in the form of three key lemmas. These lemmas capture essential properties of $R$-modules belonging to the class $\mathrm{Cot}_n$.
\begin{lemma} \label{lemma:flat-Mittag-leffler}
Let $n \geq 1$ be an integer. Then the following statements hold:
\begin{enumerate}[label=(\arabic*)]
\item Every $R$-module in $\mathcal{C}_n$ is flat and Mittag--Leffler. In particular, for any uncountable cardinal $\lambda$, every $<\lambda$-generated $R$-module $K \in \mathcal{C}_n$ is $<\lambda$-presented.
\item One has the equality $\mathrm{Cot}_n = \mathcal{C}_n^{\perp}$. Consequently, $\mathcal{C}_n \subseteq {}^{\perp}\mathrm{Cot}_n$.
\end{enumerate}
\begin{proof}
To prove~(1), observe that every $R$-module in $\mathcal{C}_n$ is a pure submodule of a free (and hence flat and Mittag--Leffler) $R$-module. Moreover, the class of flat Mittag--Leffler modules is closed under pure submodules; see \cite[Corollary~3.20(a)]{GoebelTrlifaj2012}. The final assertion then follows directly from \cite[Corollary~3.16(c)]{GoebelTrlifaj2012}.

\noindent We now prove~(2). Recall that a $R$-module $D$ belongs to $\mathrm{Cot}_n$ if and only if $\mathrm{Ext}^{\,n+1}_{R}(F,D) = 0$ for every flat $R$-module $F$ \cite[Proposition~19.2.1]{MaoDing2006CotorsionDimension}.

\noindent Let $A \in \mathcal{C}_n^{\perp}$ and let $F$ be a flat $R$-module. Take a partial free resolution of $F$ of length $n$:
\[
0 \longrightarrow V \longrightarrow F_{n-1} \longrightarrow \cdots \longrightarrow F_0 \longrightarrow F \longrightarrow 0.
\]
By construction, $V \in \mathcal{C}_n$, and there is an isomorphism $\mathrm{Ext}^1_R(V,A) \cong \mathrm{Ext}^{\,n+1}_R(F,A)$. Since $A \in \mathcal{C}_n^{\perp}$, we have $\mathrm{Ext}^1_R(V,A)=0$, hence $\mathrm{Ext}^{\,n+1}_R(F,A)=0$, showing that $A \in \mathrm{Cot}_n$.

\noindent Conversely, let $T \in \mathrm{Cot}_n$ and $K \in \mathcal{C}_n$. Then there exists an exact sequence

\[
0 \longrightarrow K \longrightarrow L_{n-1} \longrightarrow \cdots \longrightarrow L_0 \longrightarrow M \longrightarrow 0,
\]
with $L_i$ free and $M$ flat. Then $\mathrm{Ext}^1_R(K,T) \cong \mathrm{Ext}^{\,n+1}_R(M,T) = 0$, so $T \in \mathcal{C}_n^{\perp}$. This proves the equality $\mathrm{Cot}_n = \mathcal{C}_n^{\perp}$ and, as a consequence, the inclusion $\mathcal{C}_n \subseteq {}^{\perp}\mathrm{Cot}_n$.
\end{proof}
\end{lemma}

\begin{lemma}\label{lemma:Simson}
Fix an integer $n \geq 1$. Let $\lambda$ be an uncountable cardinal, and let $K \in \mathcal{C}_n$. Then there exists an exact sequence
\[
0 \longrightarrow K \longrightarrow F(I_{n-1}) \xrightarrow{d_{n-1}} F(I_{n-2}) \longrightarrow \cdots \longrightarrow F(I_{1}) \xrightarrow{d_1} F(I_{0}) \xrightarrow{d_0} M \longrightarrow 0,
\]
where, for each $r = 0,1,\ldots,n-1$, the module $F(I_r)$ denotes the free $R$-module with basis $I_r$.

\noindent Let $K'$ be any $<\lambda$-generated submodule of $K$, let $L'$ be any $<\lambda$-generated submodule of $\mathrm{Im}(d_1)$, and let $J_r \subseteq I_r$ be arbitrary subsets of cardinality $<\lambda$ for all $r = 1,2,\ldots,n-1$. Then there exist subsets $J'_r$ with $J_r \subseteq J'_r \subseteq I_r$ and $\lvert J'_r \rvert < \lambda$ for all $r = 1,2,\ldots,n-1$, together with an exact sequence
\[
0 \longrightarrow K'' \longrightarrow F(J'_{n-1}) \longrightarrow F(J'_{n-2}) \longrightarrow \cdots \longrightarrow F(J'_{1}) \longrightarrow L'' \longrightarrow 0,
\]
satisfying the following conditions:
\begin{enumerate}[label=(\arabic*)]
\item The inclusions $K' \subseteq K'' \subseteq K$ and $L' \subseteq L'' \subseteq \mathrm{Im}(d_1)$ hold, and both $K''$ and $L''$ are $<\lambda$-generated.
\item The submodule $L''$ is pure in $F(I_0)$. Consequently, $K'' \in \mathcal{C}_n$.
\end{enumerate}
\begin{proof}
This result corresponds to \cite[Proposition~1.4]{Simson1974ProjectiveResolutions}.
\end{proof}
\end{lemma}

\begin{lemma} \label{lemma:lambda-limite_C_n}
Let $n \geq 1$ be an integer. Then, for every regular uncountable cardinal $\lambda$, every $R$-module in $\mathcal{C}_n$ can be expressed as the direct limit of a $\lambda$-continuous directed system consisting of $<\lambda$-presented $R$-modules from $\mathcal{C}_n$.

\begin{proof}
We distinguish two cases. We first treat the case $n = 1$. Let $K \in \mathcal{C}_1$. Then there exists a short exact sequence
\[
0 \longrightarrow K \longrightarrow F(I_0) \longrightarrow M \longrightarrow 0,
\]
where $F(I_0)$ denotes the free $R$-module with basis $I_0$, and $M$ is a flat $R$-module.

\noindent Let $\mathcal{S}_\lambda$ be the class of all $<\lambda$-generated pure submodules of $K$. By Lemma~\ref{lemma:flat-Mittag-leffler}(1), the $R$-module $K$ is flat and Mittag--Leffler. It then follows from \cite[Corollary~3.20]{GoebelTrlifaj2012} that every subset of $K$ of cardinality $<\lambda$ is contained in a pure submodule of $K$ which is $<\lambda$-presented. Using this fact, together with the regularity of $\lambda$, we conclude that $\mathcal{S}_\lambda$, ordered by inclusion, forms a $\lambda$-continuous directed system whose direct limit is $K$. Finally, note that every module in $\mathcal{S}_\lambda$ belongs to $\mathcal{C}_1$, since the class $\mathcal{C}_1$ is closed under pure submodules, and is $<\lambda$-presented by Lemma~\ref{lemma:flat-Mittag-leffler}(1). Therefore, $K$ can be expressed as the direct limit of a $\lambda$-continuous directed system consisting of $<\lambda$-presented $R$-modules from $\mathcal{C}_1$.

\noindent We now turn to the proof of the lemma for $n \geq 2$. Let $K \in \mathcal{C}_n$. Then there exists an exact sequence
\[
0 \longrightarrow K \longrightarrow F(I_{n-1}) \xrightarrow{d_{n-1}} F(I_{n-2}) \longrightarrow \cdots \longrightarrow F(I_1) \xrightarrow{d_1} F(I_0) \xrightarrow{d_0} M \longrightarrow 0,
\]
where, for each $r = 0,1,\ldots,n-1$, $F(I_r)$ denotes the free $R$-module with basis $I_r$.

\noindent Let $\mathcal{S}_K$ denote the class of exact complexes of the form
\[
0 \longrightarrow K' \longrightarrow F(I'_{n-1}) \longrightarrow F(I'_{n-2}) \longrightarrow \cdots \longrightarrow F(I'_1) \longrightarrow M' \longrightarrow 0,
\]
where:
\begin{enumerate}[label=(\arabic*)]
\item $K'$ is a $<\lambda$-generated submodule of $K$ and $M'$ is a $<\lambda$-generated pure submodule of $\mathrm{Im}(d_1)$.
\item $F(I'_r)$ is free with basis $I'_r \subseteq I_r$, satisfying $\lvert I'_r \rvert < \lambda$ for each $r=1,2,\ldots,n-1$.
\end{enumerate}
Since $M'$ is a pure submodule of $\mathrm{Im}(d_1)$ and $\mathrm{Im}(d_1)$ is a pure submodule of $F(I_0)$, it follows that $K' \in \mathcal{C}_n$.

\noindent We define an order on $\mathcal{S}_K$ as follows. Given $\mathbb{C}, \hat{\mathbb{C}} \in \mathcal{S}_K$:
\[
\begin{aligned}
\mathbb{C} &= 0 \longrightarrow K' \longrightarrow F(I'_{n-1}) \longrightarrow F(I'_{n-2}) \longrightarrow \cdots \longrightarrow F(I'_1) \longrightarrow M' \longrightarrow 0, \\
\hat{\mathbb{C}} &= 0 \longrightarrow K'' \longrightarrow F(I''_{n-1}) \longrightarrow F(I''_{n-2}) \longrightarrow \cdots \longrightarrow F(I''_1) \longrightarrow M'' \longrightarrow 0.
\end{aligned}
\]

\noindent We say that $\mathbb{C}$ is less than or equal to $\hat{\mathbb{C}}$, and write $\mathbb{C} \leq \hat{\mathbb{C}}$, if and only if $K' \subseteq K''$, $I'_r \subseteq I''_r$ for each $r = 1, \dots, n-1$, and $M' \subseteq M''$. It is straightforward to verify, using Lemma~\ref{lemma:Simson}, that $(\mathcal{S}_K, \leq)$ is a directed system, with structure morphisms given by the inclusion maps.

\noindent To show that $(\mathcal{S}_K, \leq)$ is $\lambda$-continuous, let $\{\mathbb{C}_\alpha \mid \alpha < \theta\} \subseteq \mathcal{S}_K$ be a chain of length $\theta < \lambda$. For each $\alpha < \theta$, the complex:
\[
\mathbb{C}_\alpha = 0 \longrightarrow K'_\alpha \longrightarrow F(I^{\alpha}_{n-1}) \longrightarrow \cdots \longrightarrow F(I^{\alpha}_1) \longrightarrow M'_{\alpha} \longrightarrow 0
\]
satisfies the following properties: $K'_\alpha$ is $<\lambda$-generated, $M'_{\alpha}$ is $<\lambda$-generated and pure in $\mathrm{Im}(d_1)$, and each $F(I^\alpha_r)$ is a free module with basis $I^\alpha_r \subseteq I_r$ such that $\lvert I^\alpha_r \rvert < \lambda$. Consequently, the union of all complexes in the chain is itself the exact complex:
\[
0 \longrightarrow \bigcup_{\alpha<\theta} K'_\alpha \longrightarrow F\Big(\bigcup_{\alpha<\theta} I^{\alpha}_{n-1}\Big) \longrightarrow \cdots \longrightarrow F\Big(\bigcup_{\alpha<\theta} I^{\alpha}_1\Big) \longrightarrow \bigcup_{\alpha<\theta} M'_{\alpha} \longrightarrow 0.
\]
By the regularity of $\lambda$ and the fact that the union of any chain of pure submodules is again a pure submodule \cite[Lemma~2.25(d)]{GoebelTrlifaj2012}, this complex belongs to $\mathcal{S}_K$. 

\noindent Now, by Lemma~\ref{lemma:Simson}, the direct limit of the $\lambda$-continuous directed system $\mathcal{S}_K$ is the complex:
\[
0 \longrightarrow K \longrightarrow F(I_{n-1}) \xrightarrow{d_{n-1}} F(I_{n-2}) \longrightarrow \cdots \longrightarrow F(I_1) \xrightarrow{d_1} \mathrm{Im}(d_1) \longrightarrow 0.
\]

\noindent Consequently, $K$ can be expressed as the direct limit of a $\lambda$-continuous directed system consisting of $<\lambda$-presented $R$-modules from $\mathcal{C}_n$ (recall that, by Lemma~\ref{lemma:flat-Mittag-leffler}(1), every $R$-module in $\mathcal{C}_n$ is $<\lambda$-presented).
\end{proof}
\end{lemma}

\begin{theorem} \label{theorem:Generalization_S-S}
Let $n \geq 0$ be an integer, and let $M$ be a $n$-$\Sigma$-cotorsion $R$-module. Then every $R$-module in the definable closure of $\{M\}$ belongs to $\mathrm{Cot}_n$.
\begin{proof}
The case $n=0$ is exactly \cite[Theorem~3.3]{sarich2020singular}. Now, assume that $n \geq 1$; we proceed with an argument similar to the one used in the proof of \cite[Theorem~3.3]{sarich2020singular}.

\noindent First, we show that $\mathrm{Cogen}_{*}(M) \subseteq \mathrm{Cot}_n$. Let $K \in \mathcal{C}_n$. By Lemma~\ref{lemma:lambda-limite_C_n}, $K$ is a $\lambda$-continuous direct limit of $<\lambda$-presented $R$-modules from $\mathcal{C}_n$. Consequently, $K$ is almost $(M^{(\kappa)},\lambda)$-projective for any cardinal $\kappa$ and any regular $\lambda > \aleph_0$. Hence, by \cite[Proposition~3.2]{sarich2020singular}, $K \in {}^{\perp}\mathrm{Cogen}_{*}(M)$. It follows that $\mathcal{C}_n \subseteq {}^{\perp}\mathrm{Cogen}_{*}(M)$. Since $\mathrm{Cot}_n = \mathcal{C}_n^{\perp}$ by Lemma~\ref{lemma:flat-Mittag-leffler}(2), we conclude that $\mathrm{Cogen}_{*}(M) \subseteq \mathrm{Cot}_n$.

\noindent Finally, let $T$ be any $R$-module in the definable closure of $\{M\}$. Then there exists a short exact sequence
\[
0 \longrightarrow A \longrightarrow B \longrightarrow T \longrightarrow 0,
\]
where $A,B \in \mathrm{Cogen}_{*}(M)$. Since $\mathrm{Cogen}_{*}(M) \subseteq \mathrm{Cot}_n$ and $\mathrm{Cot}_n$ is closed under cokernels of monomorphisms (see \cite[Proposition~19.2.3]{MaoDing2006CotorsionDimension}), it follows that $T \in \mathrm{Cot}_n$.
\end{proof}
\end{theorem}

\section{\texorpdfstring{$n\text{-}\Sigma$}{n-Sigma}-cotorsion rings} \label{section:When direct sums of cotorsion have finite cotorsion dimension}

In this section we will use Theorem~\ref{theorem:Generalization_S-S} to extend  \cite[Theorem~19 and Corollary~20]{GuilAsensioHerzog2005} to characterize left $n$-perfect rings. We recall that, given $n\geq 0$, a ring $R$ is said to be \emph{left $n$-perfect} if every flat module has projective dimension at most $n$. Equivalently, the invariant  $\textrm{splf}(R)$, defined as the supremum of the projective lengths of flat $R$-modules, is $\leq n$. Notice that, as a consequence of \cite[Lemma~3.9]{cortes2016products}, if every flat $R$-module has finite projective dimension, then $R$ is left $n$-perfect for some $n\geq 0$. Notice also that $R$ is left $n$-perfect if and only if every $R$-module belongs to $\mathrm{Cot}_n$ (see, \cite[Corollary~19.2.7]{MaoDing2006CotorsionDimension}). Classical examples of left $n$-perfect rings include commutative Noetherian rings of finite Krull dimension at most $n$ and noncommutative rings $R$ of finite left Gorenstein global dimension at most $n$ (by \cite[Theorem 2.28]{Holm2004Gorenstein} and \cite[Proposition 6]{Jensen1970}). Also, by \cite{Simson1974}, every ring of cardinality $\aleph_n$ is left (and right) $(n+1)$-perfect.

\begin{definition} \label{def:n-Sigma-cotorsion-ring}
Let $n \geq 0$ be an integer. $R$ is said to be a left \emph{$n$-$\Sigma$-cotorsion ring} if $R$, regarded as a left $R$-module, is $n$-$\Sigma$-cotorsion. In the case $n=0$, $R$ is simply called a left \emph{$\Sigma$-cotorsion ring}.
\end{definition}

The following theorem shows that the class of left $n$-perfect rings coincides precisely with the class of left $n$-$\Sigma$-cotorsion rings, and it provides several equivalent characterizations of this family of rings.

\begin{theorem} \label{theorem:Generalization_Pedro_Ivo}
Let $n \geq 0$ be an integer. For the ring $R$, the following statements are equivalent:
\begin{enumerate}[label=(\arabic*)]
\item Every cotorsion $R$-module is $n$-$\Sigma$-cotorsion.
\item Any direct limit of cotorsion $R$-modules belongs to $\mathrm{Cot}_n$.
\item $R$ is left $n$-perfect.
\item The class $\mathrm{Cot}_n$ is closed under direct sums.
\item The class $\mathrm{Cot}_n$ is closed under direct limits.
\item $R$ is left $n$-$\Sigma$-cotorsion.
\item $C(R)$, the cotorsion envelope of ${_R}R$, is $n$-$\Sigma$-cotorsion.
\item Every pure-injective $R$-module is $n$-$\Sigma$-cotorsion.
\item $PE(R)$, the pure-injective envelope of ${_R}R$, is $n$-$\Sigma$-cotorsion.
\item Every direct limit of pure-injective $R$-modules belongs to $\mathrm{Cot}_n$.
\end{enumerate}
\begin{proof}
It is immediate that (3)$\Rightarrow$(2) and that (2)$\Rightarrow$(1). We now show that (1)$\Rightarrow$(3).

\noindent Since the class $\langle \mathrm{Cot} \rangle$ is definable, it admits an elementary cogenerator $D_0 \in \langle \mathrm{Cot} \rangle$. Moreover, $D_0$ is pure-injective and hence belongs to $\mathrm{Cot}$. Every $R$-module $T \in \langle \mathrm{Cot} \rangle$ is a pure submodule of a direct product of copies of $D_0$, and therefore $\langle \mathrm{Cot} \rangle = \langle D_0 \rangle$. By (1), $D_0$ is $n$-$\Sigma$-cotorsion, and thus Theorem~\ref{theorem:Generalization_S-S} yields $\langle D_0 \rangle \subseteq \mathrm{Cot}_n$. Consequently, $\langle \mathrm{Cot} \rangle \subseteq \mathrm{Cot}_n$. Since the cotorsion pair $(\mathrm{Flat}, \mathrm{Cot})$ is complete, for every $R$-module $M$ there exists a short exact sequence
\[
0 \longrightarrow M \longrightarrow C \longrightarrow F \longrightarrow 0,
\]
with $C \in \mathrm{Cot}$ and $F \in \mathrm{Flat}$. As $M$ is a pure submodule of $C$, it follows that $M \in \langle \mathrm{Cot} \rangle \subseteq \mathrm{Cot}_n$. In particular, every $R$-module has cotorsion dimension at most $n$, and hence $R$ is left $n$-perfect by \cite[Corollary~19.2.7]{MaoDing2006CotorsionDimension}.

\noindent We have thus established the equivalence of conditions $(1)$, $(2)$, and $(3)$. Furthermore, observe that (3)$\Rightarrow$(5), (5)$\Rightarrow$(4), and (4)$\Rightarrow$(1). Therefore, the first five conditions are equivalent.

\noindent Clearly, $(3)$ implies $(6)$ and $(7)$. We next show that (6)$\Rightarrow$(3). Applying Theorem~\ref{theorem:Generalization_S-S} to the $R$-module $R$, we obtain $\langle R \rangle \subseteq \mathrm{Cot}_n$. Since $\langle R \rangle = \langle \mathrm{Flat} \rangle$, it follows that every flat $R$-module belongs to $\mathrm{Cot}_n$. As the cotorsion pair $(\mathrm{Flat}, \mathrm{Cot})$ is complete, for any $R$-module $M$ there exists a short exact sequence
\[
0 \longrightarrow C \longrightarrow F \longrightarrow M \longrightarrow 0,
\]
with $C \in \mathrm{Cot}$ and $F \in \mathrm{Flat}$. Hence, $M \in \mathrm{Cot}_n$ by \cite[Proposition 19.2.3]{MaoDing2006CotorsionDimension}. Consequently, every $R$-module has cotorsion dimension at most $n$, and therefore $R$ is left $n$-perfect by \cite[Corollary~19.2.7]{MaoDing2006CotorsionDimension}.

\noindent Now, we prove that (7)$\Rightarrow$(3). Applying Theorem~\ref{theorem:Generalization_S-S} to the $R$-module $C(R)$, we obtain $\langle C(R) \rangle \subseteq \mathrm{Cot}_n$. Moreover, since $R$ is a pure submodule of $C(R)$, it follows that $\langle \mathrm{Flat} \rangle = \langle R \rangle \subseteq \langle C(R) \rangle \subseteq \mathrm{Cot}_n$. Thus, every flat $R$-module belongs to $\mathrm{Cot}_n$, and the argument proceeds exactly as in the proof that (6)$\Rightarrow$(3).

\noindent At this point, we have established the equivalence of the first seven conditions. It is clear that $(3)$ implies both $(8)$ and $(9)$. We now show that (8)$\Rightarrow$(3). As shown in the proof of (1)$\Rightarrow$(3), the class of all $R$-modules coincides with the definable closure $\langle \mathrm{Cot} \rangle$. Moreover, there exists a pure-injective $R$-module $D_0$ such that $\langle \mathrm{Cot} \rangle = \langle D_0 \rangle$. By assumption $(8)$, $D_0$ is $n$-$\Sigma$-cotorsion. Applying Theorem~\ref{theorem:Generalization_S-S}, we obtain $\langle D_0 \rangle \subseteq \mathrm{Cot}_n$. Hence, every $R$-module belongs to $\mathrm{Cot}_n$, and therefore $R$ is left $n$-perfect by \cite[Corollary~19.2.7]{MaoDing2006CotorsionDimension}.

\noindent Now, we show that (9)$\Rightarrow$(6). By assumption $(9)$, $PE(R)$ is $n$-$\Sigma$-cotorsion. Applying Theorem~\ref{theorem:Generalization_S-S}, we deduce that $\langle PE(R) \rangle \subseteq \mathrm{Cot}_n$. Since $R$ is a pure submodule of $PE(R)$ and definable classes are closed under direct sums, it follows that $R$ is a left $n$-$\Sigma$-cotorsion ring. Thus, condition $(6)$ holds.

\noindent Finally, it is clear that (10)$\Rightarrow$(8), and that (3)$\Rightarrow$(10). Hence, all the conditions are equivalent.\end{proof}
\end{theorem}

\begin{remark}
\begin{enumerate}[label=(\arabic*)]

\item Let $n \geq 0$ be an integer. By applying Lemma~\ref{lemma:nSigma_cotorsion_equiv} to the class $\mathrm{Cot}$ (or to the class of pure-injective $R$-modules), together with Theorem~\ref{theorem:Generalization_Pedro_Ivo}, one concludes that a ring $R$ is left $n$-perfect if and only if every direct sum of cotorsion $R$-modules belongs to $\mathrm{Cot}_n$. Equivalently, this holds if and only if every direct sum of pure-injective $R$-modules belongs to $\mathrm{Cot}_n$.

\item By dualizing \cite[Lemma~3.9]{cortes2016products} to the injective setting, one obtains the following result: if $R$ is a ring such that every direct sum of cotorsion $R$-modules (or every direct sum of pure-injective $R$-modules) has finite cotorsion dimension, then there exists an integer $n \geq 0$ such that the cotorsion dimension of every such direct sum is uniformly bounded above by $n$. Consequently, by applying Theorem~\ref{theorem:Generalization_Pedro_Ivo}, one concludes that $R$ is left $n$-perfect.
\end{enumerate}
\end{remark}

Specializing Theorem~\ref{theorem:Generalization_Pedro_Ivo} to the case $n=0$, we obtain the following result, which can be compared with \cite[Theorem~19 and Corollary~20]{GuilAsensioHerzog2005}.

\begin{corollary} \label{corollary:Result_of_Pedro_Ivo}
For any ring $R$, the following statements are equivalent:
\begin{enumerate}[label=(\arabic*)]
\item Every cotorsion $R$-module is $\Sigma$-cotorsion.
\item The class $\mathrm{Cot}$ is closed under direct limits.
\item $R$ is left perfect.
\item $R$ is left $\Sigma$-cotorsion.
\item $C(R)$, the cotorsion envelope of $R$, is $\Sigma$-cotorsion.
\item Every pure-injective $R$-module is $\Sigma$-cotorsion.
\item $PE(R)$, the pure-injective envelope of $R$, is $\Sigma$-cotorsion.
\item Every direct limit of pure-injective modules is cotorsion.
\end{enumerate} 
\end{corollary}

Rings satisfying the conditions of Theorem~\ref{theorem:Generalization_Pedro_Ivo} for $n=1$ are interesting, since they satisfy that pure quotients of cotorsion modules are cotorsion:

\begin{corollary}
    For the ring $R$, the following are equivalent:
    \begin{enumerate}[label=(\arabic*)]
        \item $R$ is left $1$-perfect.

        \item $\mathrm{Cot}$ is closed under pure quotients.
    \end{enumerate}
\end{corollary}

\begin{proof}
    (1) $\Rightarrow$ (2). Trivial.
    
    \noindent (2) $\Rightarrow$ (1). Let $M$ be an $R$-module. If we take its cotorsion envelope,

    \begin{displaymath}
        0 \longrightarrow M \longrightarrow C(M) \longrightarrow L \longrightarrow 0,
    \end{displaymath}
    we get, by (2), that $L$ is cotorsion. Then, for any flat $R$-module $F$ we obtain the exact sequence
    \begin{displaymath}
        \textrm{Ext}^1(F,L) \longrightarrow \textrm{Ext}^2(F,M) \longrightarrow \textrm{Ext}^2(F,C(M)) \longrightarrow 0
    \end{displaymath}
    where we deduce that $\textrm{Ext}^2(F,M)=0$, since the cotorsion pair $(\textrm{Flat},\textrm{Cot})$ is hereditary. This means that the cotorsion dimension of $M$ is less than or equal to $1$ and $R$ is left $1$-perfect.
\end{proof}

Since there exist rings which are not $1$-perfect, the class $\textrm{Cot}$ is not, in general, closed under pure quotients. However, $\textrm{Cot}$ is always closed under certain $\kappa$-pure quotients:

\begin{proposition}
    Denote by $\lambda$ the cardinal $(|R|+\aleph_0)^+$. Then $\textrm{Cot}$ is closed under $\lambda$-pure quotients.
\end{proposition}

\begin{proof}
    Let $C$ be a cotorsion $R$-module and $f:C \rightarrow L$ a $\lambda$-pure epimorphism. In view of \cite[Lemma 3.2.7]{GoebelTrlifaj2012}, $L$ will be cotorsion provided that $\textrm{Ext}^1(F,L)=0$ for every $<\lambda$-presented flat $R$-module.

    \noindent Let $F$ be a $<\lambda$-presented flat $R$-module and take a short exact sequence
    \begin{displaymath}
        0 \longrightarrow K \longrightarrow P \longrightarrow F \longrightarrow 0,
    \end{displaymath}
    with $P$ projective. For any $g:K \rightarrow L$, since $K$ is $<\lambda$-presented by Lemma \ref{lemma:flat-Mittag-leffler}(1) and $f$ is $\lambda$-pure, there exists $h:K \rightarrow C$ with $fh=g$. Since $C$ is cotorsion, $g$ has an extension, $j$, from $P$ to $C$. Then, $fj$ is an extension of $g$. Since $g$ is arbitrary, we conclude that $\textrm{Ext}^1(F,L)=0$.
\end{proof}

\section{Weakly \texorpdfstring{$n\text{-}\Sigma$}{n-Sigma}-cotorsion rings}
In this section, we consider a weaker version of left $n$-perfect rings as characterized in Theorem \ref{theorem:Generalization_Pedro_Ivo}, by considering rings $R$ for which every injective $R$-module is $n$-$\Sigma$-cotorsion, where $n \geq 0$ is an integer.

\begin{theorem} \label{theorem:Caracterization_sums_of_injectives_are_in_Cot_n}
    Let $n\geq0$ be an integer. Then, the following assertions are equivalent over any ring $R$:
    \begin{enumerate}[label=(\arabic*)]
        \item Every injective $R$-module is $n$-$\Sigma$-cotorsion.
        \item $R^{+}$, regarded as a left $R$-module, is $n$-$\Sigma$-cotorsion.
        \item The definable closure of injective $R$-modules is contained in $\mathrm{Cot}_n$.
        \item Any direct limit of FP-injective $R$-modules is in $\mathrm{Cot}_n$.
        \item Any direct limit of injective $R$-modules is in $\mathrm{Cot}_n$.
        \item Every FP-injective $R$-module belongs to $\mathrm{Cot}_n$.
    \end{enumerate}
\begin{proof}
The implication $(1)\Rightarrow(2)$ is immediate.

\noindent We now prove the implication $(2)\Rightarrow(3)$. 
Let $E$ be an injective $R$-module. Since $R^{+}$ is an injective cogenerator in the category of $R$-modules, there exists a set $\Lambda$ such that $E$ embeds into $(R^{+})^{\Lambda}$. 
Because $E$ is injective, this embedding splits, and hence $E$ is isomorphic to a direct summand of $(R^{+})^{\Lambda}$. 
In particular, $E$ is a pure submodule of $(R^{+})^{\Lambda}$, and therefore $E$ belongs to the definable closure of $R^{+}$. It follows that the definable closure of $R^{+}$ coincides with the definable closure of the class of all injective $R$-modules. 
By assumption~(2), $R^+$ is $n$-$\Sigma$-cotorsion. Consequently, by Theorem~\ref{theorem:Generalization_S-S}, the definable closure of the class of injective $R$-modules is contained in $\mathrm{Cot}_n$, which establishes~(3).

\noindent The implications $(3)\Rightarrow(4)$, $(4)\Rightarrow(5)$, and $(5)\Rightarrow(1)$ are straightforward. Thus, conditions (1)–(5) are equivalent. Finally, observe that $(3)\Rightarrow(6)$, since every FP-injective $R$-module belongs to the definable closure of injective $R$-modules, and $(6)\Rightarrow(1)$ is immediate. Therefore, all assertions are equivalent.
\end{proof}
\end{theorem}

\begin{definition} 
A ring $R$ is said to be \emph{left weakly $n$-$\Sigma$-cotorsion} if it satisfies any of the equivalent conditions of Theorem \ref{theorem:Caracterization_sums_of_injectives_are_in_Cot_n}. In case $n=0$, we say that $R$ is \emph{left weakly $\Sigma$-cotorsion}. We present several examples of this family of rings in Example~\ref{examples:weakly_Sigma-cotorsion}.
\end{definition}

\begin{remark} \label{remark:When_direcr_sum_injective_have_finite_Cot_dim}
We make the following remarks on Theorem~\ref{theorem:Caracterization_sums_of_injectives_are_in_Cot_n}.
\begin{enumerate}[label=(\arabic*)]
\item The module $R^{+}$ appearing in condition~(2) of Theorem~\ref{theorem:Caracterization_sums_of_injectives_are_in_Cot_n} may be replaced by any injective cogenerator of the category of $R$-modules. This follows from the fact that the definable closure of any injective cogenerator coincides with the definable closure of all injective $R$-modules, by the same argument used in the proof of (2) $\Rightarrow$ (3) in the previous theorem.

\item By dualizing \cite[Lemma~3.9]{cortes2016products} to the injective setting, one obtains that if $R$ is a ring for which every direct sum of injective $R$-modules has finite cotorsion dimension, then there exists an integer $n \geq 0$ such that the cotorsion dimension of every such direct sum is bounded above by $n$. Consequently, $R$ is a left weakly $n$-$\Sigma$-cotorsion ring.

\item Applying Lemma~\ref{lemma:nSigma_cotorsion_equiv} to the class of injective $R$-modules one obtain that a ring $R$ is left weakly $n$-$\Sigma$-cotorsion if and only if every direct sum of injective $R$-modules has finite cotorsion dimension at most $n$.
\end{enumerate}
\end{remark}

Specializing Theorem~\ref{theorem:Caracterization_sums_of_injectives_are_in_Cot_n} to the case $n=0$, we obtain the following result.

\begin{corollary}
\label{corollary:Caracterization_sums_of_injectives_are_cotorsion}
The following assertions are equivalent over any ring $R$:
    \begin{enumerate}[label=(\arabic*)]
        \item Every injective $R$-module is $\Sigma$-cotorsion.
        \item $R^+$, regarded as a left $R$-module, is $\Sigma$-cotorsion.
        \item The definable closure of injective modules is contained in $\mathrm{Cot}$.
        \item Any direct limit of FP-injective $R$-modules is cotorsion.
        \item Any direct limit of injective $R$-modules is cotorsion.
        \item Every FP-injective $R$-module belongs to $\mathrm{Cot}$.
    \end{enumerate}
\end{corollary}

We now describe some properties of left weakly $\Sigma$-cotorsion rings.

\begin{proposition}\label{proposition:Consequences_of_sum_of_injectives_is_cotorsion}Let $R$ be a left weakly $\Sigma$-cotorsion ring. Then the following assertions hold:
\begin{enumerate}[label=(\arabic*)]

    \item Every finitely generated flat $R$-module $F$ satisfies the following property: every ascending chain of submodules $I_{1} \subseteq I_{2} \subseteq I_{3} \subseteq \cdots$ whose union \( I := \bigcup_{n < \omega} I_{n} \) is a pure submodule of $F$ stabilizes.
    
    \item Every countably generated pure submodule of a finitely generated flat $R$-module is necessarily finitely generated.
    
    \item Every pure submodule of a finitely generated projective $R$-module is a direct summand. Consequently, every finitely generated flat $R$-module is projective.
    
    \item Every finitely generated projective $R$-module decomposes as a finite direct sum of indecomposable $R$-modules.
    
    \item $R$ has no infinite set of nonzero orthogonal idempotents; that is, $R$ is an $I$-finite ring (see \cite[p.~45]{nicholson2003quasi}).
    
    \item $R$ has the ACC on direct summands right ideals and also on direct summands left ideals.
    
    \item $R$ has the DCC on direct summands right ideals and also on direct summands left ideals.
    
\end{enumerate}

\begin{proof}
We begin with (1).  
Let $I_{1} \subseteq I_{2} \subseteq I_{3} \subseteq \cdots$ be an ascending chain of submodules of a finitely generated flat $R$-module $F$ whose union $I := \bigcup_{n < \omega} I_{n}$ is a pure submodule of $F$. Define a morphism
\[
f \colon I \longrightarrow \bigoplus_{n < \omega} \frac{F}{I_{n}}, 
\qquad 
f(x) = (x + I_{n})_{n < \omega}.
\]
Note that \( f \) is well defined, since for each \( x \in I \) the set  
\(
\{\, n < \omega \mid x \notin I_{n} \,\}
\)
is finite.

\noindent For each \( k < \omega \), let  
\(
c_{k} \colon \frac{F}{I_{k}} \longrightarrow C_{k}
\)
be a monomorphism, where $C_{k} \in \mathrm{Inj}$ (which always exists). This induces a morphism
\[
(\oplus c_{k}) \circ f \colon 
I \longrightarrow \bigoplus_{k < \omega} C_{k}.
\]

\noindent Since \( \bigoplus_{k < \omega} C_{k} \) is cotorsion (by assumption and Remark~\ref{remark:When_direcr_sum_injective_have_finite_Cot_dim}(3)), and \( I \) is pure in \( F \) (so that \( F/I \) is flat), it follows that
\[
\operatorname{Ext}^{1}_{R}\!\left(\frac{F}{I},\, \bigoplus_{k < \omega} C_{k}\right) = 0.
\]
Hence there exists a morphism
\(
h \colon F \longrightarrow \bigoplus_{k < \omega} C_{k}
\)
such that  
\(
h(y) = [(\oplus c_{k}) \circ f](y)
\)
for every \( y \in I \).

\noindent Let  \(P_{i} \colon \bigoplus_{k < \omega} C_{k} \longrightarrow C_{i}\) denote the canonical projection.  
Since \( F \) is finitely generated, write \( F = \langle x_{1},\dots,x_{n} \rangle \).  
Choose \( m < \omega \) such that  
\(
P_{m}(h(x_{i})) = 0 \text{ for all } i = 1,\dots,n.
\)
Then \( P_{m}(h(x)) = 0 \) for all \( x \in F \).

\noindent We claim that \( I_{k} = I_{m} \) for every \( k \geq m \).  
Indeed, let \( x \in I_{k} \). On the one hand, \( P_{m}(h(x)) = 0 \).  
On the other hand,
\[
P_{m}(h(x))
= P_{m}([(\oplus c_{k}) \circ f](x))
= P_{m}\!\left( (\oplus c_{k})(x + I_{n})_{n < \omega} \right)
= P_{m}\!\left( (c_{n}(x + I_{n}))_{n < \omega} \right)
= c_{m}(x + I_{m}).
\]
Thus \( c_{m}(x + I_{m}) = 0 \), and since \( c_{m} \) is a monomorphism, we conclude \( x \in I_{m} \).  
Hence the chain stabilises at \( m \), establishing (1).

\noindent We now prove (2).  
Let \( K \) be a countably generated pure submodule of a finitely generated flat $R$-module \( F \).  
Choose generators \( \{ x_{n} \mid n < \omega \} \) for \( K \), and define  
\(
K_{0} := 0,\,
K_{1} := Rx_{1},\,\dots,\,
K_{n} := \sum_{i=1}^{n} Rx_{i}.
\)
Then
\(
K_{0} \subseteq K_{1} \subseteq K_{2} \subseteq \cdots
\)
is an ascending chain with  
\(
K = \bigcup_{n < \omega} K_{n}.
\)
Applying (1), this chain must stabilise at some \( m < \omega \).  
Thus \( K = \sum_{i=1}^{m} Rx_{i} \), proving that \( K \) is finitely generated.

\noindent Now, we prove (3). We begin by showing that every countably generated pure submodule of a finitely generated projective $R$-module is a direct summand, and is therefore projective. Let \( M \) be a countably generated pure submodule of a finitely generated projective $R$-module \( P \).  
By (2), the $R$-module \( M \) is finitely generated.  
Consider the short exact sequence
\[
0 \longrightarrow M \longrightarrow P \longrightarrow P/M \longrightarrow 0.
\]
Since \( P \) is finitely generated projective and \( M \) is finitely generated, it follows that \( P/M \) is a finitely presented flat $R$-module.  
Hence \( P/M \) is projective by \cite[Proposition~3.2.12]{EnochsJenda2011}.  
Therefore the sequence splits, and so \( M \) is a direct summand of \( P \).  
In particular, \( M \) is projective.

\noindent Now, let \( K \) be a pure submodule of a finitely generated projective $R$-module \( P \). We claim that $K$ is countably generated. Once this is established, the desired conclusion follows from the previous case. Suppose that $K$ is not countably generated. Fix $x_{0} \in K$. Since $K$ is flat and Mittag-Leffler then by \cite[Corollary 3.20(d)]{GoebelTrlifaj2012} there exists a countably generated pure submodule $P_0$ of $K$ containing $x_{0}$. Since $P_{0} \subsetneq K$, there exist $x_{1} \in K-P_{0}$. Applying again that $K$ is flat and Mittag-Leffler, there exist a countably generated pure submodule $P_{1}$ of $K$ containing $P_{0}+Rx_{1}$. Proceeding in this way, we get an infinite ascending chain $P_{0} \subsetneq P_{1} \subsetneq \cdots$ of pure submodules of $P$ that does not terminate, which is a contradiction with (1).

\noindent
To prove (4), we simply combine \cite[Lemma~2.16]{MohamedMuller1990} and \cite[Theorem~2.17]{MohamedMuller1990}.

\noindent
Finally, to prove (5), (6), and (7), note that (1) implies (5) by \cite[Proposition~6.59]{Lam1999}, and (5), (6), and (7) are equivalent because the proof of \cite[Proposition~6.59]{Lam1999} works for both left and right ideals.
\end{proof}
\end{proposition}

In what follows, we use the results established so far to derive several relationships between the class of left weakly $n$-$\Sigma$-cotorsion rings and other well-known classes of rings.

\begin{corollary} \label{corollary:Relation_with_semiperfect_rings}
The following assertions hold for any ring $R$:
\begin{enumerate}[label=(\arabic*)]
  \item If $R$ is a left weakly $\Sigma$-cotorsion ring and cotorsion as a left $R$-module, then $R$ is semiperfect.
  \item If $R$ is a left weakly $\Sigma$-cotorsion ring and flat covers of finitely generated $R$-modules are finitely generated, then $R$ is left almost-perfect (see \cite[Definition~3.2]{amini2008rings}). In particular, $R$ is semiperfect.
\end{enumerate}
\begin{proof}
We first prove~(1). By Proposition~\ref{proposition:Consequences_of_sum_of_injectives_is_cotorsion}(5), the ring $R$ is $I$-finite. Since, by hypothesis, $R$ is cotorsion, it follows from \cite[Theorem~3.19]{mao2006cotorsion} that $R$ is semiperfect.

\noindent We now prove~(2). By Proposition~\ref{proposition:Consequences_of_sum_of_injectives_is_cotorsion}(3), every finitely generated flat $R$-module is projective. Hence, under the hypothesis that flat covers of finitely generated $R$-modules are finitely generated, it follows that such flat covers are projective. Therefore, $R$ is left almost-perfect by \cite[Theorem~3.7]{amini2008rings}.
\end{proof}
\end{corollary}

Recall that a ring $R$ is called a \emph{left $\mathrm{IF}$-ring} (see \cite{colby1975rings}) if every injective left $R$-module is flat. The following corollary shows that, for left $\mathrm{IF}$-rings, the class of left Noetherian rings coincides with the class of left weakly $\Sigma$-cotorsion rings.

\begin{corollary} \label{corollary:Relation_with_IF_rings}
Let $R$ be a left $\mathrm{IF}$-ring. Then the following assertions are equivalent:
\begin{enumerate}[label=(\arabic*)]
    \item $R$ is left Noetherian.
    \item $R$ is left weakly $\Sigma$-cotorsion.
\end{enumerate}
\begin{proof}
Clearly, (1) implies (2). We now prove that (2) implies (1).

\noindent Let $\{E_i \mid i \in I\}$ be a family of injective $R$-modules. Consider the pure-exact sequence
\[
0 \longrightarrow \bigoplus_{i \in I} E_i \longrightarrow \prod_{i \in I} E_i \longrightarrow T \longrightarrow 0.
\]

\noindent Since $R$ is a left $\mathrm{IF}$-ring, the product $\prod_{i \in I} E_i$ is flat. Because the sequence is pure, the $R$-module $T$ is also flat. By (2) and Remark~\ref{remark:When_direcr_sum_injective_have_finite_Cot_dim}, $\bigoplus_{i \in I} E_i$ belongs to $\mathrm{Cot}$, so $\mathrm{Ext}^1_R\big(T, \bigoplus_{i \in I} E_i\big) = 0$. Hence, the short exact sequence splits, and $\bigoplus_{i \in I} E_i$ is injective. Consequently, $R$ is left Noetherian.
\end{proof}
\end{corollary}

A ring $R$ is said to be \emph{left absolutely pure}, or \emph{left $\mathrm{FP}$-injective}, if $R$, viewed as a left $R$-module over itself, is $\mathrm{FP}$-injective. In Theorem~\ref{theorem:Generalization_Pedro_Ivo}, we characterized left $n$-$\Sigma$-cotorsion rings in terms of the cotorsion envelope (or the pure-injective envelope) of the ring being $n$-$\Sigma$-cotorsion.
It is therefore natural to ask what happens if the injective envelope of the ring is $n$-$\Sigma$-cotorsion.
With this motivation, we obtain the following result:

\begin{corollary} \label{corollary:Relation_with_FP-Rings}
Let $R$ be a left absolutely pure ring. Then the following assertions are equivalent:
\begin{enumerate}[label=(\arabic*)]
    \item $E(R)$, the injective envelope of $R$, is $n$-$\Sigma$-cotorsion for some $n \ge 0$.
    \item $R$ is left $m$-perfect for some $m \ge 0$.
    \item $R$ is left weakly $r$-$\Sigma$-cotorsion for some $r \ge 0$.
\end{enumerate}
Moreover, if any (and hence all) of the above conditions hold, then $n = m = r$.
\begin{proof}
\noindent (1)$\Rightarrow$(2). Since $R$ is $\mathrm{FP}$-injective as an $R$-module, it follows that $R$ is a pure submodule of $E(R)$. Consequently, $R \in \langle E(R) \rangle$. On the other hand, by (1) and Theorem~\ref{theorem:Generalization_S-S}, we obtain that $\langle E(R) \rangle \subseteq \mathrm{Cot}_n$. Since definable classes are closed under direct sums, it follows that $R$ is left $n$-$\Sigma$-cotorsion. Therefore, by Theorem~\ref{theorem:Generalization_Pedro_Ivo}, $R$ is left $n$-perfect.

\noindent The implications (2)$\Rightarrow$(3) and (3)$\Rightarrow$(1) are straightforward. Hence, all the conditions are equivalent.
\end{proof}
\end{corollary}

\begin{remark}\label{Remark:Relation_with_Von-Neumann-Rings}
    \begin{enumerate}[label=(\arabic*)]
        \item Let $R$ be a von Neumann regular ring. Then $R$ has finite left global dimension if and only if it is a left weakly $r$-$\Sigma$-cotorsion ring for some $r \ge 0$. Indeed, since every $R$-module over a von Neumann regular ring is $\mathrm{FP}$-injective, Corollary~\ref{corollary:Relation_with_FP-Rings} applies, taking into account that a von Neumann regular ring has finite left global dimension if and only if it is left $m$-perfect for some $m \geq 0$. Consequently, a von Neumann regular ring with infinite left global dimension cannot be a left weakly $n$-$\Sigma$-cotorsion ring for any $n \ge 0$. An explicit example is given in Example~\ref{example:Nagata}(2).

        \item Let $R$ be an indiscrete ring (see \cite{prest1995absolutely}). By \cite[Proposition 2.3]{prest1995absolutely}, $R$ is both a left and right $\mathrm{FP}$-injective, so Corollary~\ref{corollary:Relation_with_FP-Rings} applies to this class of rings. Notice that there are indiscrete rings that are not von Neumann regular (see \cite[Section 2.2]{prest1995absolutely}).
    \end{enumerate}
\end{remark}

At this point, we present several examples of left weakly $n$-$\Sigma$ cotorsion rings.

\begin{example}\label{examples:weakly_Sigma-cotorsion}
    \begin{enumerate}[label=(\arabic*)] 
    \item All left Noetherian rings and all left $n$-perfect rings are left weakly $n$-$\Sigma$-cotorsion rings. This follows immediately from the definition. Note that if we consider the product of a ring that is not left Noetherian but is left $n$-perfect with a ring that is not left $n$-perfect but is left Noetherian, then we obtain a ring that is neither left Noetherian nor left $n$-perfect, yet still left weakly $n$-$\Sigma$-cotorsion ring. An example illustrating this situation for $n=0$, is the ring $T := \mathbb{Z} \times R^{op}$, where $R$ is the ring constructed in~\cite[Example~1]{estrada2017gorenstein}, $\mathbb{Z}$ denotes the ring of integers, and $R^{op}$ denotes the opposite ring of $R$.

    \item Let $R$ be any left Noetherian and not left perfect ring, and $Q$ an (infinite) rooted quiver (in the sense of \cite[Definition 1]{Oyonarte2001CotorsionQuivers}), such that the path ring $RQ$ is not left Noetherian. For instance, by \cite[Lemma 3.4 and Theorem 3.6]{EnochsEtAl2002_NoetherianQuivers}, $Q$ can be the infinite binary tree quiver:
    
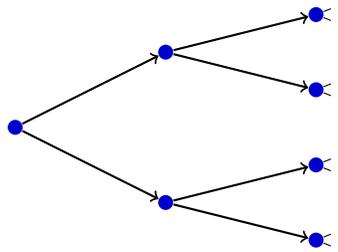
\begin{center}
      \begin{tikzpicture}[
         node distance = 1cm,
    dot/.style = {circle, fill=blue!80!black, inner sep=2pt}]

    \node[dot] (root) at (0,0) {};

    \node[dot] (n1) at (2, 1) {};
    \node[dot] (n2) at (2, -1) {};

    \node[dot] (n11) at (4, 1.5) {};
    \node[dot] (n12) at (4, 0.5) {};
    \node[dot] (n21) at (4, -0.5) {};
    \node[dot] (n22) at (4, -1.5) {};

\draw[->, thick] (root) -- (n1);
\draw[->, thick] (root) -- (n2);

\draw[->, thick] (n1) -- (n11);
\draw[->, thick] (n1) -- (n12);
\draw[->, thick] (n2) -- (n21);
\draw[->, thick] (n2) -- (n22);

\draw[dashed, shorten >=5pt] (n11) -- (4.5, 1.7);
\draw[dashed, shorten >=5pt] (n11) -- (4.5, 1.3);

\draw[dashed, shorten >=5pt] (n12) -- (4.5, 0.7);
\draw[dashed, shorten >=5pt] (n12) -- (4.5, 0.3);

\draw[dashed, shorten >=5pt] (n21) -- (4.5, -0.3);
\draw[dashed, shorten >=5pt] (n21) -- (4.5, -0.7);

\draw[dashed, shorten >=5pt] (n22) -- (4.5, -1.3);
\draw[dashed, shorten >=5pt] (n22) -- (4.5, -1.7);

\end{tikzpicture}
    \end{center}

\noindent Then, there is a family of injective left $RQ$-modules, $\{E_i|\ i\in I\}$, such that $\bigoplus_{i\in I}E_i$ is not injective. However, as $R$ is left Noetherian, by \cite[Theorem 6]{Oyonarte2001CotorsionQuivers}, every direct sum of injective left $RQ$-modules is cotorsion. Therefore $RQ$ is a left weakly $\Sigma$-cotorsion ring by Remark~\ref{remark:When_direcr_sum_injective_have_finite_Cot_dim}(3).

\item In \cite{Kaplansky1958_Hereditary}, Kaplansky constructed a ring $R$ that is von Neumann regular, not left Noetherian, and has left global dimension exactly~$2$. Consequently, the class of injective left $R$-modules coincides with the class of cotorsion $R$-modules. Since $R$ is not left Noetherian, there exists a family of injective left $R$-modules whose direct sum fails to be injective and, therefore, is not cotorsion in this context, then $R$ is not a left weakly $\Sigma$-cotorsion ring by Remark~\ref{remark:When_direcr_sum_injective_have_finite_Cot_dim}(3). Nevertheless, every left $R$-module has cotorsion dimension at most~$2$. Thus, $R$ is a left weakly $2$-$\Sigma$-cotorsion ring, but it is not a left weakly $\Sigma$-cotorsion ring.
\end{enumerate}
\end{example}

We now prove a result connecting left $n$-perfect and left weakly $n$-$\Sigma$-cotorsion rings. 
Note that, when $n=0$, the statement below recovers \cite[Lemma~5.1(1)]{EstradaPerezZhu2020}, so the result can be regarded as a generalization of that lemma. Here, by a balanced pair we mean the same concept as in \cite[Introduction]{EstradaPerezZhu2020}; see also \cite[Section~8.2]{EnochsJenda2011}.

\begin{corollary} \label{corollary:Generalization_lemma_5.1(1)}
Let $n \geq 0$ be an integer. For any ring $R$, the following statements are equivalent:
\begin{enumerate}[label=(\arabic*)]
\item $R$ is left weakly $n$-$\Sigma$-cotorsion, and the class $^{\perp}\!\mathrm{Cot}_n$ forms the left part of a balanced pair.
\item $R$ is a left $n$-perfect ring (equivalently, $R$ is a left $n$-$\Sigma$-cotorsion ring by Theorem~\ref{theorem:Generalization_Pedro_Ivo}).
\end{enumerate}
\begin{proof}
For each integer $n \geq 0$, the pair $\bigl(^{\perp}\!\mathrm{Cot}_n, \mathrm{Cot}_n\bigr)$ is a cotorsion pair by \cite[Lemma 1.13]{Christensen2023OneSidedGorenstein}. Hence, the implication $(1)\Rightarrow(2)$ is proved by the same argument as in \cite[Lemma~5.1(1)]{EstradaPerezZhu2020}, by using Theorem \ref{theorem:Generalization_Pedro_Ivo}.

\noindent We now establish $(2)\Rightarrow(1)$. 
Since $R$ is left $n$-perfect, every $R$-module has cotorsion dimension at most $n$. Therefore, the class $^{\perp}\!\mathrm{Cot}_n$ coincides with the class of all projective $R$-modules and thus forms the left component of a balanced pair, whose right component is given by the class of all injective $R$-modules.
\end{proof}
\end{corollary}

Finally, we provide examples illustrating that the class of left $n$-perfect rings is strictly contained within the class of left weakly $n$-$\Sigma$-cotorsion rings. We also demonstrate that there are rings which are not left weakly $n$-$\Sigma$-cotorsion, for any $n\geq 0$. 

\begin{example}\label{example:Nagata}
\begin{enumerate}[label=(\arabic*)]
\item In a private communication, Keri Ann Sather-Wagstaff and Jonathan P. Totushek explained how to modify Nagata's example (see \cite[Appendix A1, Example 1]{Nagata1962_LocalRings}) to construct a commutative Noetherian domain $R$ whose field of fractions has infinite projective dimension. Consequently, $R$ is not $n$-perfect for any $n \geq 0$. Nevertheless, since $R$ is Noetherian, every direct sum of injective $R$-modules is cotorsion, and therefore $R$ is left weakly $n$-$\Sigma$-cotorsion, for all $n\geq 0$ by Remark~\ref{remark:When_direcr_sum_injective_have_finite_Cot_dim}(3).

\item Let $k$ be a field, and let $R$ be the direct product of $\aleph_{\omega}$ copies of $k$. By construction, $R$ is a von Neumann regular ring. Moreover, by \cite[Conclusions~3.1]{Osofsky1970_Homological}, the ring $R$ has infinite global dimension. Consequently, $R$ is not left $n$-$\Sigma$-cotorsion ring for any $n \geq 0$ by Remark~\ref{Remark:Relation_with_Von-Neumann-Rings}(1).

\noindent Since every von Neumann regular ring is left coherent, the ring $R$ provides an example of a left coherent ring that is not left weakly $n$-$\Sigma$-cotorsion for any $n \geq 0$. Furthermore, let $T := S^{\mathrm{op}} \times R$, where $S$ is the ring constructed in~\cite[Example~1]{estrada2017gorenstein}. Then $T$ is left weak coherent (in the sense of \cite{cortes2016products}), but it is neither left coherent nor left weakly $n$-$\Sigma$-cotorsion for any $n \geq 0$.
\end{enumerate}
\end{example}

\subsection{Weakly \texorpdfstring{$n\text{-}\Sigma$}{n-Sigma}-cotorsion triangular matrix rings}

In this part, we determine when triangular matrix rings are left weakly $n$-$\Sigma$-cotorsion for some $n \ge 0$. We will adopt the notation established in \cite{Mao2020CotorsionPairs}. Let $A$ and $B$ be rings, and let $U$ be a $(B,A)$-bimodule. Consider the formal lower triangular matrix ring $T =\begin{pmatrix}A & 0 \\U & B\end{pmatrix}$. Recall that every $T$-module can be represented as a pair $M =\begin{pmatrix}M_1 \\M_2\end{pmatrix}_{\varphi^{M}}$, where $M_1$ is an $A$-module, $M_2$ is a $B$-module, and $\varphi^{M} \colon U \otimes_{A} M_1 \longrightarrow M_2$ is a $B$-homomorphism. Analogously, every right $T$-module $N$ can be represented as a pair $N =\begin{pmatrix}N_1 \\N_2\end{pmatrix}_{\tau^{N}}$, where $N_1$ is a right $A$-module, $N_2$ is a right $B$-module, and $\tau^{N} \colon N_2 \otimes_{B} U \longrightarrow N_1$ is an $A$-homomorphism. Injective, projective, and flat $T$-modules are completely characterized; see, for instance, \cite[Lemma~3.1]{Mao2020CotorsionPairs}. Cotorsion $T$-modules are also characterized in \cite[Corollary~4.3]{Mao2020CotorsionPairs}. 

\begin{remark} \label{remark:caracterization_Cot_n_in_triangular_matrix_rings}
Let $n\geq0$ be an integer. It is straightforward to verify, using \cite[Proposition~19.2.1]{MaoDing2006CotorsionDimension} and \cite[Lemma~3.2]{Mao2020CotorsionPairs}, that a $T$-module $M = \begin{pmatrix} M_1 \\ M_2 \end{pmatrix}_{\varphi^{M}}$ has cotorsion dimension at most $n$ if and only if the $A$-module $M_1$ and the $B$-module $M_2$ each have cotorsion dimension at most $n$. In particular, $M$ is $n$-$\Sigma$-cotorsion if and only if both $M_1$ and $M_2$ are $n$-$\Sigma$-cotorsion and the same conclusions hold for right modules as well.
\end{remark}

\begin{proposition} \label{proposition:Caracterization_Sums_of_Inj_is_in_Cot_n_for_triangular_matrix_rings}
Let $n \geq 0$ be an integer. The following statements are equivalent:
\begin{enumerate}[label=(\arabic*)]
    \item $T$ is a left weakly $n$-$\Sigma$-cotorsion ring.
    \item $A$ and $B$ are left weakly $n$-$\Sigma$-cotorsion rings and the $A$-module $U^{+}$ is $n$-$\Sigma$-cotorsion.
\end{enumerate}
\begin{proof}
The regular module $T_T$ can be described as the pair $\begin{pmatrix}
A \oplus U \\
B
\end{pmatrix}_{\tau}$, where $\tau \colon B \otimes_{B} U \longrightarrow A \oplus U$ is the $A$-homomorphism defined by $\tau(b \otimes u) = (0, bu)$, for all $b \in B$ and $u \in U$. Consequently, the $T$-module $T^{+}$ is represented by the pair $\begin{pmatrix}
A^{+} \oplus U^{+} \\
B^{+}
\end{pmatrix}_{\tau^{+}}$, where $\tau^{+} \colon U \otimes_{A} (A^{+} \oplus U^{+}) \longrightarrow B^{+}$ is given by $\tau^{+}(u \otimes m^{+})(b) := m^{+}(\tau(b \otimes u))$, for all $u \in U$, $m^{+} \in A^{+} \oplus U^{+}$, and $b \in B$. Therefore, applying Remark~\ref{remark:caracterization_Cot_n_in_triangular_matrix_rings} and Theorem~\ref{theorem:Caracterization_sums_of_injectives_are_in_Cot_n} yields the desired result.
\end{proof}
\end{proposition}

\begin{remark} \label{remark:triangular_matrix_rings}
Let $n \geq 0$ be an integer. We now describe several situations in which the $A$-module $U^{+}$ is $n$-$\Sigma$-cotorsion, under the assumption that the rings $A$ and $B$ are left weakly $n$-$\Sigma$-cotorsion. First, observe that for every injective $B$-module $E$, the $A$-module $\operatorname{Hom}_{B}(U, E)$ is cotorsion. Moreover, it is easy to verify that there exists an $A$-module isomorphism $U^{+} \cong \operatorname{Hom}_{B}(U, B^{+})$, where $B^{+}$ is regarded as a left $B$-module. Consequently, to ensure that $U^{+}$ is $n$-$\Sigma$-cotorsion, it suffices (under the assumption that both $A$ and $B$ are left weakly $n$-$\Sigma$-cotorsion rings) to verify one of the following additional hypotheses. Under any of these conditions, the ring $T$ is also left weakly $n$-$\Sigma$-cotorsion by 
Proposition~\ref{proposition:Caracterization_Sums_of_Inj_is_in_Cot_n_for_triangular_matrix_rings}.

\begin{enumerate}[label=(\arabic*)]
    \item \textbf{$A$ is left $n$-perfect.}  
    In this case every $A$-module has finite cotorsion dimension at most $n$. Consequently, the $A$-module $U^{+}$ is $n$-$\Sigma$-cotorsion.

    \item \textbf{$U$ is flat as a right $A$-module.}  
    Then $U^+$ is an injective $A$-module. Since $A$ is a left weakly $n$-$\Sigma$-cotorsion, then $U^+$ is $n$-$\Sigma$-cotorsion.

    \item \textbf{$B$ is left Noetherian and $U$ is finitely generated as a $B$-module.} The finite generation of $U$ yields, for any set $I$, an isomorphism $\operatorname{Hom}_{B}(U, B^{+})^{(I)} \cong \operatorname{Hom}_{B}\!\left(U, (B^{+})^{(I)}\right)$. Since $B$ is left Noetherian, the $B$-module $(B^{+})^{(I)}$ is injective, and consequently $\operatorname{Hom}_{B}\!\left(U, (B^{+})^{(I)}\right)$ is cotorsion. Then, $U^+$ is $r$-$\Sigma$-cotorsion for all integer $r\geq0$.
\end{enumerate}
\end{remark}

\begin{example} \label{examples_with_matrix_rings}
We now provide, in the case $n=0$, an example of a ring satisfying each of the three conditions of Remark~\ref{remark:triangular_matrix_rings} while failing the other two. It is clear that each of the three rings presented below is left weakly $\Sigma$-cotorsion.
\begin{enumerate}[label=(\arabic*)]
    \item Consider the ring $\begin{pmatrix} R & 0 \\ R^{\Lambda} & R \end{pmatrix}$, where $R$ is the ring described in~\cite[Example~1]{estrada2017gorenstein}. Since $R$ is left perfect, condition~(1) is satisfied. However, because $R$ is not right coherent, Chase's theorem \cite[Theorem 2.1]{Chase1960} guarantees the existence of a set $\Lambda$ such that the $R$-module $R^{\Lambda}$ is not flat; hence, condition~(2) fails. Moreover, since $\Lambda$ must be infinite, condition~(3) is also not satisfied.
    \item Consider the ring $\begin{pmatrix} \mathbb{Z} & 0 \\ \mathbb{Q} & \mathbb{Z} \end{pmatrix}$. This ring clearly satisfies condition~(2), since $\mathbb{Q}$ is a flat $\mathbb{Z}$-module. However, it does not satisfy condition~(1), as $\mathbb{Z}$ is not left perfect, nor condition~(3), because $\mathbb{Q}$ is not finitely generated as a $\mathbb{Z}$-module. Note that this ring is neither left Noetherian nor left perfect.
    \item Let $m>1$ be an integer, and consider the ring $\begin{pmatrix} \mathbb{Z} & 0 \\ \mathbb{Z}_m & \mathbb{Z}_m \end{pmatrix}$. Here, $\mathbb{Z}_m$ denotes the ring of integers modulo $m$. This ring clearly satisfies condition~(3). However, it does not satisfy condition~(1), since $\mathbb{Z}$ is not left perfect, nor condition~(2), because $\mathbb{Z}_m$ is not a flat right $\mathbb{Z}$-module.
\end{enumerate}
\end{example}
\section*{Acknowledgements} We wish to thank Lars Winther Christensen for sharing with us the notes of Sather-Wagstaff and Totushek, which lead to Example~\ref{example:Nagata}(1). The authors also thank Pedro Antonio Guil Asensio for helpful correspondence during the preparation of this manuscript.
\printbibliography[heading=bibintoc]
\end{document}